\documentclass[12pt,a4paper]{amsart}
\usepackage{amsfonts}
\usepackage{amscd}
\usepackage[latin1]{inputenc}
\usepackage{t1enc}
\usepackage[mathscr]{eucal}
\usepackage{indentfirst}
\usepackage{graphicx}
\usepackage{graphics}
\usepackage{pict2e}
\usepackage{epic}
\numberwithin{equation}{section}
\usepackage[margin=2.8cm]{geometry}
\usepackage{epstopdf} 
\usepackage{amsmath}
\allowdisplaybreaks
\usepackage{amssymb}
\usepackage{amsthm}
\usepackage{bbm}
\usepackage{tikz}
\usepackage{paralist}
\usepackage[onehalfspacing]{setspace}
\usepackage{amsfonts}
\usepackage{esint}
\usepackage{color}
\usepackage{nicefrac}
\usepackage{mathrsfs}
\usepackage{multirow}

\newcounter{dummy}
\usepackage{hyperref}
\usepackage{enumitem}
\makeatletter
\newcommand\myitem[1][]{\item[#1]\refstepcounter{dummy}\def\@currentlabel{#1}}
\makeatother
\newtheorem{thm}{Theorem}
\numberwithin{thm}{section}
\newtheorem{lem}[thm]{Lemma}

\newtheorem{defi}[thm]{Definition}

\newtheorem{thmintro}{Theorem}

\newtheorem*{thm*}{Theorem}
\newtheorem*{prop*}{Proposition}
\numberwithin{equation}{section}
\theoremstyle{remark}
\newtheorem{rem}[thm]{Remark}

\newcommand{\V}{\vert}

\newcommand{\R}{\mathbb{R}}
\newcommand{\N}{\mathbb{N}}

\newcommand{\LL}{\mathcal{L}}

\newcommand{\conv}{\mathrm{conv}}

\newcommand{\dx}{\,\textup{d}x}
\newcommand{\dt}{\,\textup{d}t}
\newcommand{\dy}{\,\textup{d}y}
\newcommand{\dz}{\,\textup{d}z}
\newcommand{\ds}{\,\textup{d}s}
\newcommand{\dr}{\,\textup{d}r}

\newcommand{\Cor}{\mathrm{Cor}}

\newcommand{\dH}{\,\textup{d}\mathcal{H}}
\newcommand{\dL}{\,\textup{d}\mathcal{L}}

\newcommand{\Haus}{\mathcal{H}}
\newcommand{\dmu}{\,\textup{d}\mu}

\renewcommand{\phi}{\varphi}

\everymath{\displaystyle}

\DeclareMathOperator{\curl}{curl}
\DeclareMathOperator{\skw}{skew}

\DeclareMathOperator{\dist}{dist}

\DeclareMathOperator{\divergence}{div}

\DeclareMathOperator{\sgn}{sgn}
\DeclareMathOperator{\Lip}{Lip}
\DeclareMathOperator{\Sim}{Sim}

\newcommand{\M}{\mathscr{M}}

\definecolor{Gump}{rgb}{0,0.6,0.4}
\definecolor{Hanks}{rgb}{0.7,0.3,0.1}

\begin{document}
\title[Solenoidal truncation]{An alternative approach to solenoidal Lipschitz truncation}
\author[Schiffer]{Stefan Schiffer}
\address{Insitute for Applied Mathematics, University of Bonn} 
\email{schiffer@iam.uni-bonn.de}
 \subjclass[2010]{}
 \keywords{} 
\begin{abstract}
In this work, a new approach to obtain a solenoidal Lipschitz truncation is presented. More precisely, the goal of the truncation is to modify a function $u \in W^{1,p}(\mathbb{R}^3,\mathbb{R}^3)$ that satisfies the additional constraint $\mathrm{div}~  u=0$, such that its modification $\tilde{u}$ is in $W^{1,\infty}(\mathbb{\R}^3,\mathbb{R}^3)$ and still is divergence-free. We give an alternative approach to Lipschitz truncation compared to previous works by Breit, Diening \& Fuchs \cite{BDF} and Breit, Diening \& Schwarzacher \cite{BDS}. The ansatz pursued here allows a rather strict bound on the $W^{1,p}$ distance of $u$ and $\tilde{u}$.

\end{abstract}
\maketitle
\section{Introduction} \label{sec:intro}
\subsection{Lipschitz extensions and truncations}
A technique that is important both in functional analytic results and in applications to partial differential equations (PDEs) is \emph{Lipschitz extension}. More precisely, take a metric space $(X,d)$, a closed subset $Y \subset X$ and a function $u \colon Y \to \R^d$ that is Lipschitz continuous, i.e. there is $L>0$ such that 
    \[
        \V u(x) -u(y) \V \leq L d(x,y) \quad \forall x,y \in Y.
    \]
The aim of Lipschitz extension is to find a function $\tilde{u} \colon X \to \R^d$ that coincides with $u$ on $Y$ and still is Lipschitz continuous on $X$ with the same Lipschitz constant (or a constant that is only worse by some additional multiplicative constant).
Such an extension result has been achieved by \textsc{McShane} and \textsc{Kirszbraun} \cite{McS,Kirszbraun} and, in a slightly different setting, by \textsc{Whitney} \cite{Whitney,Stein}.

\emph{Lipschitz truncation} has a different focus. For simplicity, take $X=\R^N$ with the euclidean metric. Given a function $u \colon \R^N \to \R^d$ we want to modify $u$, such that its modification is Lipschitz continuous and (in some sense) close to $u$. The main strategy to achieve such a truncation relies on the observation that $u$ is already Lipschitz continuous when restricting to a rather large set $X$. We then keep $u$ unchanged on this set and perform a Lipschitz extension on its complement. In the context of $L^1$-functions, this large set can be described via the maximal function $\mathcal{M}$, i.e. 
    \[
    \V u(x) - u(y) \V \leq C(N) L \V x- y \V
    \]
whenever $x,y$ are contained in the set $X=\{ \mathcal{M}(D u) \leq L\}$, cf. \cite{Liu,AF}.

\subsection{Whitney's approach to extension and truncation}

The approach to Lipschitz extension/truncation that we further pursue in this work is the truncation that is due to \textsc{Whitney} \cite{Whitney}, which has been refined in the works \cite{AF,AF2,Zhang,DMS}. Given a closed subset $X \subset \R^N$ and a function $u \colon X \to \R^d$, we find a Calder\'{o}n-Zygmund decomposition (a composition into dyadic cubes) of its complement, $X^C= \bigcup_{i \in \N} Q_i$. 

\noindent For each cube $Q_i \subset X^C$, we may find a \emph{projection point} $x_i \in X$ close to the cube. Given a partition of unity $\phi_i$ of $X^C$, with $\phi \in C_c^{\infty}(Q_i)$, we define 
    \begin{equation} \label{intro:truncWhitney}
        \tilde{u}(y) = \left\{ \begin{array}{ll} \sum_{i \in \N} \phi_i(y) u(x_i) & y \in X^C \\ u(y) & y \in X \end{array} \right.
    \end{equation}
One can show (cf. \cite{Stein}), that this is indeed a Lipschitz extension from $X$ to $\R^N$. If the Lipschitz constant of $u$ restricted to $X$ is $L$, then the Lipschitz constant of $\tilde{u}$ is no worse than $C(N,d) L$ for a purely dimensional constant $C(N,d)$. In particular, $C(N,d)$ does \emph{not} depend on $X$.

The technique that is presented below in the context of \emph{solenoidal} Lipschitz truncation/extension, heavily relies on Whitney's truncation and a modification of formula \eqref{intro:truncWhitney}. 

\subsection{Solenoidal truncations}

The truncation becomes harder if we want to preserve an additional constraint, in particular $\divergence u =0$. The truncated version of $u$ shall still satisfy the constraint $\divergence \tilde{u}=0$. 
This $\divergence$-free or \emph{solenoidal} Lipschitz truncation has been extensively discussed in the context of fluid dynamics, e.g. \cite{BDF,BDS,DSS,NR,MS,BR}. Moreover, such a truncation result can be used to derive lower-semicontinuity result for functionals of the form
    \[
        I(u)=\int f(x,u(x),Du(x))
    \] 
together with an additional constraint $\divergence u=0$, cf. \cite{AF} for the unconstrained setting.

We now shortly outline the two approaches to solenoidal truncations, which are revisited in more detail in Section \ref{sec:old}. For simplicity, we stick to the physical dimension $N=3$.

The first approach, that has been advocated in \cite{BDS}, is to write a divergence-free function $u \in W^{1,p}(\R^3,\R^3)$ as 
    \[
    u = \curl U
    \]
for a function $U \in W^{2,p}(\R^3,\R^{3\times3}_{\skw})$. Then one performs a $W^{2,p}$-$W^{2,\infty}$-truncation (a higher order Lipschitz truncation) to $U$ and then sets $\tilde{u} = \curl \tilde{U}$. While this technique and the ensuing result is sufficient in the context of fluid mechanics, there is one major issue. This result does not obtain good bounds on the distance between $\tilde{u}$ and $u$. In particular, even if $u \in W^{1,\infty}$, the function $U$ is in general \emph{not} in $W^{2,\infty}$ (cf. \cite{Ornstein,DLM}). Consequently, even $u \in W^{1,\infty}$, the truncated version $\tilde{u}=\curl \tilde{U}$ that we obtain through this approach \emph{never} coincides with $u$.

The second approach, that has been executed in a slightly different setting in \cite{BDF} (for truncating symmetric gradients instead of full gradients), is more accurate. The idea is to start with \eqref{intro:truncWhitney}. This $\tilde{u}$ does not satisfy the constraint $\divergence \tilde{u}=0$. Therefore, one adds corrector terms $\mathrm{Cor}_i$ that, up to an additive constant, satisfy
    \[
    \divergence \mathrm{Cor}_i = \divergence  \left(\phi_i \sum_{j \in \N} \phi_j u(x_j)\right).
    \]
This corrector term is obtained through the \emph{Bogovski\u{\i}-operator} that solves the equation $\divergence v =f$ in a cube $Q_i$. As this technique directly works with the function $u$ and not with $U$ satisfying $\curl U= u$, this result leads to stronger bounds on the distance between $u$ and its truncated version $\tilde{u}$.

\subsection{A third approach and the main truncation result} \label{sec:intro:3}

In this paper, we pursue a third approach, that is related to the second one, but skips using the Bogovski\u{\i} operator. This approach might be useful to deal with other differential constraints or with other geometries (for example on a Riemannian manifold) than the Euclidean geometry. Note that the assumption that the geometry is flat is heavily used through bounds for the Bogovski\u{i} operator in the second approach described above.

The idea still is to obtain a divergence-free truncation from \eqref{intro:truncWhitney} by adding suitable corrector terms that achieve solenoidality. Our ansatz, however, directly gives these terms using cancellations and the fact that $\phi_i$ is a partition of unity. It turns out that we can write the truncation in the following form:
    \[
        \tilde{u}(y) = \left\{ \begin{array}{ll} \sum_{i \in \N} \phi_i(y) u(x_i) + \sum_{i,j \in \N} \phi_j D \phi_i A(i,j) + \sum_{i,j,k\in \N} \phi_k D \phi_j D \phi_i B(i,j,k) & y \in X^C, \\ u(y) & y \in X, \end{array} \right.
    \]
for suitable terms $A(i,j)$ and $B(i,j,k)$ depending on $u$ (cf. Definition \ref{defi:trunc}).

This is used to obtain the following truncation statement. 
\begin{thmintro}[Solenoidal Lipschitz truncation] \label{thmintro:A}
Let $u \in W^{1,p}(\R^3,\R^3)$ satisfy the differential constraint $\divergence u=0$. For any $L>1$, there exists a nonlinear truncation operator $T_L$ acting on $W^{1,p}(\R^3,\R^3)$ and a constant $C$, such that \begin{enumerate} [label=(\alph*)]
    \item $\divergence T_L u=0$;
    \item $\Vert T_L u \Vert_{W^{1,\infty}} \leq C L$;
    \item \label{A:3} $\Vert T_L u - u \Vert_{W^{1,p}}^p \leq C \int_{\{\V u \V \geq L\} \cup \{ \V D u\V \geq L\}} \V u \V^p + \V Du \V^p \dx$;
    \item \label{A:4} $\mathcal{L}^3 ( \{ u \neq T_L u \}) \leq C  L^{-p} \int_{\{\V u \V \geq L\} \cup \{ \V D u\V \geq L\}} \V u \V^p + \V Du \V^p \dx$.
\end{enumerate}
\end{thmintro}

As mentioned, we restrict ourselves to the physically relevant dimension $N=3$. While pursuing the approach is still possible in higher dimension, the truncation (that relies on cancellation and the application of Stokes' theorem) gets slightly complicated. Therefore, we stick to $N=3$. We shortly discuss the dimension $N=2$ in Section \ref{sec:N2}.

Finally, we shall also mention that the question of truncation is also relevant for other grades of regularity. In particular, the question of solenoidal $L^1$-$L^{\infty}$-truncation  (instead of $W^{1,1}$-$W^{1,\infty}$-truncation in the present paper) has been addressed in \cite{Schiffer,BGS} with a very similar approach to the problem as in the present paper. Moreover, one may also weaken the starting space to be $\mathrm{BV}(\R^3,\R^3) \cap \ker \divergence$ instead of $W^{1,1}(\R^3,\R^3) \cap \ker \divergence$ (cf. \cite{BDG} for the unconstrained truncation). To keep this work at a reasonable length, we however stick to the regularity given by Theorem \ref{thmintro:A} and maps from $\R^3$ to $\R^3$.

\subsection{Structure of this article}
The remainder of this article is split into two further sections. In Section \ref{sec:Whitney} we revisit the classical Whitney's extension/truncation theorem and its modern variant with the additional constraint of solenoidality. In more detail, we recall the construction of a Whitney cover in Section \ref{sec:WC} and recall its application in form of Whitney's extension and truncation in Section \ref{sec:WE}. In Section \ref{sec:old} we recall two main methods to obtain solenoidal Lipschitz truncations: The potential truncation, following \cite{BDS} and a truncation via local corrections, following \cite{BDF}.

In Section \ref{sec:trunc} we present our approach to divergence-free Lipschitz truncation. Section \ref{sec:truncdefi} is concerned with the definition of the truncation and an outline of the main steps of the proof of Theorem \ref{thmintro:A}. The proofs are then carried out in Sections \ref{sec:intermezzo} and \ref{sec:proofs}.

 \subsection*{Acknowledgements}
 The author would like to thank Franz Gmeineder and Sergio Conti for helpful discussions. The author acknowledges the support of the Deutsche Forschungsgemeinschaft (DFG, German Research Foundation) through the graduate school BIGS of the Hausdorff Center for Mathematics (GZ EXC 59 and 2047/1, Projekt-ID 390685813).
\section{Whitney's truncation theorem and its application to (solenoidal) Lipschitz truncation} \label{sec:Whitney}

In this section, we revisit Whitney's classical approach to Lipschitz truncation extension that is based on a Calder\'{o}n-Zygmund decomposition of a set. Then we shortly discuss and compare the approaches to \emph{solenoidal} Lipschitz truncation from \cite{BDF} and \cite{BDS} and outline differences to the ansatz pursued in this paper. Finally, we give a short argument, why the question of soloneoidal Lipschitz truncation in dimension $N=2$ is rather easy; essentially this follows from the fact that $\curl$ and $\divergence$ are the same operation in two dimension after a suitable rotation. Therefore, we focus on $N=3$ in Section \ref{sec:trunc}.

\subsection{Whitney cubes} \label{sec:WC}

Our goal is to extend a function from a closed set $X \subset \R^N$ to $\R^N$. For truncation this means that we leave the function $u \colon \R^N \to \R^N$ unchanged and modify it on $X^C$. We often refer to $X$ as the `good set' and to its complement as the `bad set'.

Following \textsc{Stein}'s book \cite[Chapter VI]{Stein}, we can cover the bad set $X^C$ by open, dyadic cubes $(Q_i^{\ast})_{i \in \N}$ with the following properties: \begin{enumerate} [label=(\roman*$^{\ast}$)]
    \item $X^C = \bigcup_{i \in \N} \bar{Q}_i^{\ast}$;
    \item $Q_i^{\ast} \cap Q_j^{\ast} = \emptyset$ if $i \neq j$;
    \item $\tfrac{1}{4} \dist(Q_i^{\ast},X) \leq l(Q_i^{\ast}) \leq 4\dist(Q_i^{\ast},X)$ where $l(Q_i^{\ast})$ denotes the sidelength of $Q_i^{\ast}$;
    \item if $\bar{Q}_i^{\ast} \cap \bar{Q}_j^{\ast} \neq \emptyset$, then $\tfrac{1}{4} l(Q_i^{\ast}) \leq l(Q_j^{\ast}) \leq 4 l(Q_i^{\ast})$;
    \item for each $i \in \N$, the number of cubes $Q_j^{\ast}$ such that $\bar{Q}_i^{\ast} \cap \bar{Q}_j^{\ast} \neq \emptyset$ is bounded by a dimensional constant $C(N)$.
\end{enumerate}
Furthermore, for each cube $Q_i^{\ast}$, we denote by $c_i$ the \emph{centre} of the cube $Q_i^{\ast}$. We may also find a \emph{projection point} $z_i \in X$, such that $\dist(Q_i^{\ast},x_i) \leq 4 \dist(Q_i^{\ast},X)$. Moreover, we define the measure $\mu_i$ as \[
 \mu_i = \mathcal{L}^N(1/2 Q_i^{\ast})^{-1} \bigl( \mathcal{L}^N \llcorner (\tfrac{1}{2}Q_i)\bigr)
\]
where $\tfrac{1}{4} Q_i^{\ast}$ is the open cube with axis parallel faces, centre $c_i$ and sidelength $\tfrac{1}{4}l(Q_i^{\ast})$.

In adddition, we also consider slightly blown-up open cubes $Q_i=(1+\epsilon) Q_i^{\ast}$ with the same centre $c_i$ but sidelength $(1+\varepsilon)l(Q_i^{\ast})$. If $\epsilon$ is sufficiently small (e.g. $\varepsilon<\tfrac{1}{32}$), then these cubes have the following properties:
\begin{enumerate}[label=(\roman*)]
    \item $X^C = \bigcup_{i \in \N} Q_i$;
    \item for all $i \in \N$, the number of cubes $Q_j$ with $Q_j \cap Q_i \neq \emptyset$ is bounded by $C(N)$;
    \item $\tfrac{1}{5} \dist(Q_i,X) \leq l(Q_i) \leq 5\dist(Q_i,X)$, where $l(Q_i)$ is the sidelength of $Q_i$ ($l(Q_i) =(1+\epsilon)l(Q_i^{\ast})$;
    \item if $Q_i \cap Q_j \neq \emptyset$, then $\tfrac{1}{4} l(Q_i) \leq l(Q_j) \leq 4 l(Q_i)$.
\end{enumerate}

We now take $\phi \in C_c^{\infty}((-\varepsilon/2,1+\varepsilon/2)^N)$ with $\phi \equiv 1$ on $[0,1]^N$. By translation and scaling we get $\phi_j^{\ast} \in C_c^{\infty}(Q_j)$ with $\phi^{\ast}_j \equiv 1$ on $Q_j^{\ast}$.

Using the properties of the cubes $Q_j$ we can show (again, cf. \cite{Stein}), that 
    \begin{equation} \label{def:phi}
    \phi_j = \frac{\phi_j^{\ast}}{\sum_{i \in \N} \phi_i^{\ast}}
    \end{equation}
defines a partition of unity on $X^C$, i.e. 
    \[
    \sum_{j \in \N} \phi_j(y) = \chi_{X^C}(y) := \left\{ \begin{array}{ll} 1 & \text{on } X^C, \\ 0 & \text{on } X. \end{array} \right.
    \]
Moreover, $\phi_j \in C_c^{\infty}(Q_j)$ and they satisfy the bound
    \begin{equation} \label{est:phiD}
    \Vert D^k \phi_j \Vert_{L^{\infty}} \leq C(N,k) l(Q_j)^{-k}.
    \end{equation}
    
Before we continue with Whitney's extension theorem, we shall point out again that all the dimensional constants $C(N)$ and $C(N,k)$ in above construction \emph{do not} depend on the regularity of the set $X$; all we need is the set $X$ to be closed.    
    
\subsection{Whitney extension and the related truncation theorem} \label{sec:WE}
Given a closed set $X \subset \R^N$ and a Lipschitz continuous function $u \colon X \to \R$ with Lipschitz constant $L>0$, we define 
    \begin{equation} \label{sec2:Wt}
        \tilde{u}(y) = \left\{ \begin{array}{ll} \sum_{i \in \N} \phi_i(y) u(z_i)& y \in X^{C}, \\ u(y) & y \in X. \end{array} \right. 
    \end{equation}
\begin{lem}[Whitney truncation]
If $u \colon X \to \R$ is Lipschitz continuous with Lipschitz constant $L$, then $\tilde{u} \colon \R^N \to \R$ is Lipschitz continuous with Lipschitz constant $C_N L$ for a purely dimensional constant $C_N$.
 \end{lem}
A detailed proof can be found in \cite[Chapter VI]{Stein}. We shall, however, give a brief heuristic idea, why this works; a similar argumentation is then used in Section \ref{sec:trunc}.

Observe that $\tilde{u} \in C^{\infty}(X^C)$ and that locally, only finitely many terms are nonzero. Therefore, we can compute its derivative for $y \in X^C$: \begin{align*}
    D \tilde{u}(y) &= \sum_{i \in \N} D \phi_i(y) u(x_i) = \sum_{i,j \in \N} \phi_j D\phi_i(y) (u(x_i)-u(x_j))
\end{align*}
Note that we explicitly used  in the second step that $(\phi_i)_{i \in \N}$ and $(\phi_j)_{j \in \N}$ are partitions of unity. Applying the bound for the derivative of $\phi_i$ and the Lipschitz bound for $u(x_i)-u(x_j)$ yields an $L^{\infty}$-bound for the derivative of $\tilde{u}$.

Together with the observation, that $\tilde{u}$ is continuous, this leads to Lipschitz continuity of $\tilde{u}$.

\medskip

This approach can then be applied to Lipschitz \emph{truncation} as follows: First, we find a `good set' $X$, on which $X_{\lambda}$ is Lipschitz continuos. For the truncation we then take $\tilde{u} \equiv u$ on the good set and redefine is on the bad set as in \eqref{sec2:Wt}.

So, an important part in Lipschitz truncation is the following Lemma \ref{lem:AF} that connects Lipschitz continuity to the (centered) Hardy-Littlewood maximal function $\M u$. 

Recall that this function is defined as \[
\M u(x) := \sup_{r>0} \fint_{B_r(x)} \V u(z) \V \dz,
\]
where $\fint$ denotes the average integral. The Hardy-Littlewood maximal function has the following properties:
\begin{lem}[The maximal function] \label{lem:maximalf}
The maximal function is sublinear, i.e. $M(u+v) \leq Mu + Mv$. Moreover, it is bounded as a map from $L^p(\R^N) \to L^p(\R^N)$, whenever $1<p\leq \infty$. For $p=1$, it is not bounded from $L^1(\R^N)$ to $L^1(\R^N)$, but from $L^1(\R^N)$ to $L^{1,\infty}(\R^N)$, i.e.
    \[
    \mathcal{L}^N( \{ \V u \V \geq \lambda \}) \leq C \lambda^{-1} \Vert u \Vert_{L^1}
    \]
for every $\lambda>0$.    
\end{lem}

Key to the truncation statement then are the following two observations. The first one, proven in \cite{Liu,AF} proves that a function $u$ is Lipschitz continuous on sublevel sets of its maximal function.
\begin{lem} \label{lem:AF}
Let $u \in W^{1,p}(\R^N,\R^d)$ be continuous. There exists a dimensional constant $C$, such that for all $\lambda>0$ and all $x,y \in X_{\lambda}= \{ \M (Du) \leq \lambda\}$ we have 
    \begin{equation} \label{AF:estimate}
        \vert u(x) - u(y) \vert \leq C \lambda \vert x- y \vert.
    \end{equation}
\end{lem}

The second ingredient is an estimate on the measure of the `bad set', the complement of $X_{\lambda}$, cf. \cite{Zhang}.
\begin{lem}
Let $v \in L^p(\R^N,\R^d)$, $1 \leq p< \infty$. There is a dimensional constant $C$, such that for all $\lambda>0$ we have 
    \begin{equation} \label{Zhang:set:est}
        \LL^N(\{ \M v > \lambda \}) \leq C \lambda^{-p} \int_{\{ u \geq \lambda\}} \vert v  \vert^p \dx.
    \end{equation}
\end{lem}

These results are combined to get a \emph{Lipschitz truncation} in the fashion we already mentioned: We take $u \colon \R^N \to \R^d$, set $\tilde{u}=u$ on the `good set' $X_{\lambda}$ and then extend this restriction via Lipschitz extension \eqref{sec2:Wt} to the full space.

The definition \ref{sec2:Wt} only has a slight problem: The value $u(x_i)$ might not be well-defined a priori (unless $W^{1,p}$ embeds into the continuous functions). One can circumvent this in the following ways:
\begin{enumerate} [label=(\roman*)]
    \item Define $\tilde{u}$ first for smooth function and then use a density argument (cf. \cite{Schiffer}).  The point here is that both the smooth functions and their truncation converge pointwise almost everywhere on a set with small complement;
    \item Only choose Lebesgue points as projection points $z_i$;
    \item replace the use of the value at a single point $u(z_i)$ by an average of $u$.
\end{enumerate}

In the following, we pursue the third approach, which is easy to adopt to related problems. Therefore, we define:
    \begin{equation} \label{eq:Lipequation}
        T_{\Lip} u(y) = \left\{ \begin{array}{ll} \sum_{i \in \N} \phi_i(y) \int u(x_i) \dmu_i(x_i) & y \in X_{\lambda}^C,\\
         u(y) & y \in X_{\lambda}. \end{array} \right.
    \end{equation}
A suitable adaptation of Lemma \ref{lem:AF} and the estimate \eqref{Zhang:set:est} then yield the following:

\begin{lem}[Lipschitz truncation] \label{lem:Lipschitztrunc}
Let $u \in W^{1,p}(\R^N,\R^d)$. Define the set $X_{\lambda} = \{ \M u \leq \lambda\} \cup \{\M (Du) \leq \lambda\}$. Then the truncation $T_{\Lip} u$ as in \eqref{eq:Lipequation} has the following properties:
\begin{enumerate} [label=(\alph*)]
    \item the function $T_{\Lip} u$ is Lipschitz continuous and $\Vert T_{\Lip} u \Vert_{W^{1,\infty}} \leq C \lambda$;
    \item \label{lem:Lipschitz:2} the set $X_{\lambda}^C \supset \{ u \neq T_{\Lip} u\}$ has $\LL^N$-measure bounded by 
        \[
            \LL^N(X_{\lambda}^C) \leq C \lambda^{-p} \int_{\{u\geq \lambda\} \cup \{Du \geq \lambda\}} \vert u \vert^p + \vert D u\vert^p \dx;
        \]
    \item as a consequence of \ref{lem:Lipschitz:2} we have
        \[
            \Vert u - T_{\Lip} u \Vert_{W^{1,p}}^p \leq \int_{\{u\geq \lambda\} \cup \{Du \geq \lambda\}} \vert u \vert^p + \vert D u\vert^p \dx.
        \]
\end{enumerate}
\end{lem}

\subsection{Solenoidal truncation so far} \label{sec:old}

In this subsection, we discuss two approaches to solenoidal Lipschitz truncation. In contrast to classical Lipschitz truncation, we take $u \in W^{1,p}(\R^N,\R^N)$ satisfying the additional constraint $\divergence u =0$. We look for a truncated version $\tilde{u}$, that is not only Lipschitz, but also satisfies the additional constraint $\divergence \tilde{u}=0$. There is no reason, why the truncation defined as in \eqref{sec2:Wt} should satisfy this constraint, hence we need to take another definition.

For simplicity, we shall assume that the space dimension is $N=3$. In dimension $N=2$, the situation is quite different (cf. Section \ref{sec:N2}).
\subsubsection{Potential truncation} \label{sec:potential}
First, we revisit the potential truncation, as advocated in \cite{BDS} (also see \cite{Gallenm,BGS} for related discussions). We focus on truncations of functions on the torus $T_3$ instead of $\R^3$, as the case for the full space is rather similar. The idea is that we can write $u \in W^{1,p}(T_3,\R^3)$ with zero average as
    \[
        u = \curl U
    \]
for a function $U \in W^{2,p}(T_3,\R^3)$ if and only if $\divergence u=0$.

We now define 
    \[
    S U (y)= \left\{ \begin{array}{ll}
                    \sum_{i \in \N} \phi_i(y) \int U(x_i)+ DU(x_i) \cdot (y-x_i) \dmu_i(x_i) & y \in X^C, \\
                    U(y) & y \in X. \end{array} \right.
     \]
Here, the set $X$ is defined as 
    \[
    X = \{\M U \leq \lambda \} \cap \{\M (DU) \leq \lambda \} \cap \{\M (D^2U) \leq \lambda\}
    \]
The operator $S$ is a $W^{2,p}$-$W^{2,\infty}$-truncation, i.e. a higher order Lipschitz truncation (cf. \cite{Stein,BGS}). Hence, $SU$ has the following properties: \begin{enumerate} [label=(\roman*)]
    \item $SU \in W^{2,\infty}(T_3,\R^3)$ and $\Vert SU \Vert_{W^{2,\infty}} \leq \lambda$;
    \item $\LL^3( \{ U \neq SU \}) \leq C \lambda^{-p} \int_{Y_{\lambda}} \vert U \vert^p + \vert D U \vert^p +\vert D^2 U \vert^p \dx$ for the set     \[
        Y_{\lambda} = \{ \vert U \vert \geq \lambda\} \cup \{ \vert DU \vert \geq \lambda\} \cup \{ \vert D^2U \vert \geq \lambda\}.
        \]
    \item $\Vert U - SU \Vert_{W^{1,p}}^p \leq C \int_{Y_{\lambda}} \vert U \vert^p + \vert D U \vert^p +\vert D^2 U \vert^p \dx$.  
\end{enumerate}

We then define the solenoidal truncation of $u$ as 
    \[
        \tilde{u} = \curl SU.
    \]
Consequently, $\tilde{u}$ also inherits the properties of $SU$ . The problem is, however, that in dimension $N=3$, the sets $X$ and $Y_{\lambda}$ have measures bounded in terms of $U$ and \emph{not} in terms of $u$. In particular, even if $u \in W^{1,\infty}(T_3,\R^3)$, its potential $U$ is not in $W^{2,\infty}(T_3,\R^3)$ in general (Ornstein's non-inequality \cite{Ornstein}). Therefore, the set $X$ might be non-empty for any $\lambda>0$ in this case. In particular, a bound of the type 
    \[
         \LL^3(\{ u \neq \tilde{u} \}) \leq C \lambda^{-p} \int_{\{\vert u \vert \geq \lambda \} \cup \{ \vert D u \vert \geq \lambda\}} \vert u \vert^p + \vert Du \vert^p \dx
    \]
is not achievable through this ansatz. To get a \emph{strong statement} as in Theorem \ref{thmintro:A}, we need a more careful approach.

\subsubsection{Truncation via local corrections} \label{sec:loccor}

The second approach that follows \cite{BDF}, tackles the problem outlined for the potential truncation. Instead of writing $u=\curl U$, we work with $u$ directly, and try to add corrector terms, to modify truncation $\tilde{u}$ in \eqref{sec2:Wt} to be solenoidal. As the `good set' one takes the set where the maximal function of $u$ and of $Du$ is small\footnote{In the original paper \cite{BDF} the authors only considered the symmetric part $Eu$ of the gradient, but the approach stays the same.}.
 
In more detail, we first add a global corrector term $\Pi u$, such that \begin{enumerate} [label=(\roman*)]
    \item $\tilde{u} + \Pi u =: T_0 u$ is still a $W^{1,\infty}$-truncation of $u$;
    \item $\int_{Q_i} \phi_i (\tilde{u}+\Pi u) =0$.
\end{enumerate}
In the second step, one adds local corrector terms $\Cor_i \in W^{1,\infty}(Q_i,\R^3)$, that satisfy the divergence-equation
    \begin{equation} \label{eq:diveq}
    \divergence \Cor_i = \divergence (\phi_i(\tilde{u}+\Pi u).
    \end{equation}
Finally, one then obtains the divergence-free truncation via 
    \[
    T_{\divergence} u = T_0 u - \sum_{i \in \N} \Cor_i.
    \]
As one only modifies the function on the bad set $\{ M u > \lambda \} \cup \{ M(Du) >\lambda\}$, one is able to obtain the bounds achieved in Theorem \ref{thmintro:A}. One crucial observation (cf. \cite[Lemma 2.11]{BDF}) to get $T_{\divergence} u \in W^{1,\infty}$ is the following:

First note, that $\Cor_i \in W^{1,\infty}$ is \emph{not} trivial, as the divergence equation \eqref{eq:diveq} and its solution operator (the Bogovski\u{\i} operator \cite{Bogovski}) is \emph{not} bounded from $L^{\infty}$ to $W^{1,\infty}$. Instead, one uses that $\divergence (\phi_i(\tilde{u}+\Pi u))$ has a very specific form. It is a much smoother function, hence in $W^{1,\infty}$. A \emph{uniform} $W^{1,\infty}$-bound is achieved via the following argument. The partition of unity is obtained by formula \eqref{def:phi} for a covering of cubes. But actually, as all cubes $Q_i^{\ast}$ are dyadic, there are only finitely many configurations how the cover can locally look like. In particular, up to scaling and translation, we only need to solve 
    \[
        \divergence \Cor_i = \divergence (\phi_i(\tilde{u}+\Pi u)).
    \]
a \emph{finite} amount of times as, up to scaling and translation, there are only finitely many functions the right-hand-side can realise. As we then only argue about \emph{finitely} many different corrector terms, the uniform $W^{1,\infty}$-bound is an easy consequence.

\subsubsection{Truncations via local corrections II}

In this work, we study a modification of the second approach that is still able to obtain the strong bounds of Theorem \ref{thmintro:A}, but may be more flexible than using the second. In particular, our approach only relies on some geometric estimates and the existence of a cover with sets $Q_i$ and a partition of unity $\phi_i$ that satisfies estimate \eqref{est:phiD}; these sets do not need to be cubes. In particular, although we only consider the Euclidean geometry, the approach presented here is in principle extendable to Riemannian manifolds and, in general, the operator of exterior differentiation instead of the divergence operator.

In particular, we also add corrector terms, such that $\tilde{u}$ form \eqref{sec2:Wt} is modified to be divergence-free. Instead of solving the divergence-equation, we however give explicit formulas for the corrector terms. Moreover, we make use of certain cancellations and do not define the corrections on single cubes, but on tuples of cubes, i.e. 
    \[
        T_{\divergence} u = \tilde{u} + \sum_{i,j \in \N} \Cor_{i,j} + \sum_{i,j,k \in \N} \Cor_{i,j,k}.
    \]
The exact definition of these corrector terms and the corresponding calculations are part of Section \ref{sec:trunc}.

\subsection{Solenoidal truncation in $N=2$} \label{sec:N2}

In this subsection, we discuss the solenoidal truncation in space dimension $N=2$. Here, the situation is hugely different, as the constraint $\divergence u=0$ is a much more restricting condition than in dimension $N=3$.

Actually, the potential truncation ansatz outlined in paragraph \ref{sec:potential} works in this space dimension. We now outline, how this concretely is achieved, as this also gives an idea for $N=3$. As in \ref{sec:potential}, we concentrate on functions on the torus first.

\smallskip

Note that we can rewrite any function $u\in L^2(T_2,\R^2)$ as 
    \[
        v=\left( \begin{array}{c} v_1 \\ v_2 \end{array} \right) 
        := \left( \begin{array}{c} u_2 \\ -u_1 \end{array} \right).
    \]
Then the constraint $\divergence u =0$ translates to $\curl v=0$. Now note that for any $v \in W^{1,p}(T_2,\R^2)$ with zero average satsfying $\curl v=0$, we can find $V \in W^{1,p}(T_2)$ such that 
    \[
        v = D V.
    \]
The huge contrast to $N=3$ is that this operation is bounded from $W^{1,p}$ to $W^{2,p}$ for \emph{all} $p \in [1,\infty]$, in particular $p=\infty$ is well included. Moreover, $D V=v$ and $D^2 V =DV$. Hence, the truncation for $V$ given by 
    \[
        T V (y)= 
            \left\{     
                \begin{array}{ll} \sum_{i \in \N} \phi_i(y) \int V(x_i) + D V(x_i) \cdot (y-x_i) \dmu_i(x_i) & y \in X^C, \\
                V(y) & y \in X,
                \end{array}
            \right.
    \]
yields a $\curl$-free $W^{1,p}$-$W^{1,\infty}$ truncation by setting $\tilde{v}:= D (TV)$ and the set 
    \[
        X:= \{ \M D V \leq \lambda \} \cap \{ \M D^2 V \leq \lambda \} =  \{ \M v \leq \lambda \} \cap \{ \M Dv \leq \lambda \} .
    \]
\medskip

An important observation is that we can actually write $\tilde{v} = D (TV)$ in terms of $v$ instead of $V$ (so, in principle, we may skip the step of finding such a $V$). For (almost every) $y \in X$, we get $\tilde{v}= DV = v$. So we concentrate on $y \in X^C$. As $\phi_i$ are supported on cubes $Q_i$, the sum is locally finite hence:
    \begin{align*}
        \tilde{v} = D(TV) &= \sum_{i\in \N} \phi_i(y) \int DV(x_i) \dmu_i(x_i) \\
        & + \sum_{i \in \N} D \phi_i(y) \int V(x_i)+DV(x_i) \cdot (y-x_i) \dmu_i(x_i).
    \end{align*}
The first term of this sum equals the usual Lipschitz truncation, as $DV=v$; \begin{align*}
            \sum_{i\in \N} \phi_i(y) \int DV(x_i) \dmu_i(x_i) = \sum_{i\in \N} \phi_i(y) \int v(x_i) \dmu_i(x_i) 
        \end{align*}    
        
For the second term we use that $\phi_i$ is a partition of unity to introduce a second index:        
        \begin{align*}
            &(\ast) :=\sum_{i \in \N} D \phi_i(y) \int (V(x_i)+DV(x_i) \cdot (y-x_i) \dmu_i(x_i) \\&= \sum_{i,j \in \N} \phi_j D \phi_i  \iint (V(x_i)+DV(x_i) \cdot (y-x_i)) - (V(x_j)+DV(x_j) \cdot (y-x_j)) \dmu_i \dmu_j.
        \end{align*}  
In the second term we may now use fundamental theorem applied to the function     \[
    z \longmapsto V(z) + DV(z) \cdot (y-z).
    \]
and obtains    \begin{align*}
    (\ast) = \sum_{i,j\in \N} \phi_j D \phi_i \fint_{[x_i,x_j]} \iint (x_i-x_j) \cdot D^2 V(z) \cdot (y-z) \dmu_i(x_i) \dmu_j(x_j) \dH^1(z). \end{align*}
As $D^2 V =Dv$ one gets the $\curl$-free Lipschitz truncation:
    \begin{equation} \label{eq:2Dtrunc}
        \tilde{v}(y) = \left\{  
                            \begin{array}{ll} \sum_{i \in \N} \phi_i(y) \int v(x_i) \dmu_i(x_i) + \sum_{i,j \in \N} \Cor_{i,j} & y \in X^C, \\
                            v(y) & y \in X,
                            \end{array}
                        \right.
    \end{equation}
  where the corrector term $\Cor_{i,j} \in C^{\infty}(Q_i \cap Q_j,\R^2)$ is defined as 
    \begin{equation}   \label{eq:2Dcor}
  \Cor_{i,j} =\sum_{i,j\in \N} \phi_j D \phi_i \fint_{[x_i,x_j]} \iint (x_i-x_j) \cdot Dv(z) \cdot (y-z) \dmu_i(x_i) \dmu_j(x_j) \dH^1(z).
    \end{equation}  
    
\begin{rem}
\begin{enumerate} [label=(\alph*)]
    \item The definition of $\tilde{v}$ in \eqref{eq:2Dtrunc} is independent of its potential. So formula \eqref{eq:2Dtrunc} also gives a truncation for functions which are in $W^{1,p}(\R^2,\R^2)$.
    \item When deriving \eqref{eq:2Dtrunc}, we never used that the space dimension is $N=2$, only that $\curl v=0$. Indeed, \eqref{eq:2Dtrunc} gives a $\curl$-free Lipschitz truncation in any dimension.
    
\end{enumerate}
\end{rem}
\section{Solenoidal Lipschitz truncation} \label{sec:trunc}

This section is concerned with the proof of Theorem \ref{thmintro:A}. Its proof is split into several independent steps.

As the proof is quite involved we proceed as follows. First, we give the definition of the truncation and the main lemmas to obtain Theorem \ref{thmintro:A}. Then, we gather some preliminary results which are necessary to prove these lemmas (cf. Section \ref{sec:intermezzo}). Finally, we give the proofs in Section \ref{sec:proofs}.

\subsection{Definition of the truncation} \label{sec:truncdefi}

Let $X \subset \R^d$ be a closed set and $(Q_i)_{i \in \N}$ be a Whitney cover of $X^C$ featuring a partition of unity $\phi_i$ and measures $\mu_i$ (cf. Section \ref{sec:WC}). For simplicity, we write 
    \[
   \dmu_{i,j} := \dmu_i(x_i) \dmu_j(x_j), \quad \dmu_{i,j,k}= \dmu_i(x_i) \dmu_j(x_j) \dmu_k(x_k)
    \]
etc. For points $x_{i_1},\ldots,x_{i_k} \in \R^3$ denote by \[
\Sim(i_1,\ldots,i_k) = \Sim(x_{i_1},\ldots,x_{i_k)} := \conv (x_{i_1},...,x_{i_k})
\]
the convex hull of those points. If $k=2$, we often denote this by the direct line $[x_{i_1},x_{i_2}]:=\Sim(i_1,i_2)$ between these points. If $k=2,3,4$, we call this object a non-degenerate simplex if its dimension is $(k-1)$, i.e. $\Haus^{k-1}(\Sim(i_1,\ldots,i_k)))>0 $.

\begin{defi} \label{defi:trunc}
Given $u \in W^{1,p}(\R^3,\R^3)$, tuples $(a,b,c) \in \{(1,2,3),(2,3,1),(3,1,2)\}$ and $X \subset \R^d$, we define the truncation operator $T$ as follows: \begin{equation} \label{trunc:1}
    (T u (y)) _a= \left\{ 
                    \begin{array} {ll}
                        \sum_{i \in \N} \phi_i(y) \int u_a(x_i) \dmu_i(x_i) + (S u(y))_a + (Ru(y))_a & y \in X^C \\
                        u (y) & y \in X 
                    \end{array} 
                \right.
\end{equation}
where the error terms $S$, $R$ are defined below. First, let us write 
    \[
        T_0 u := T u - S u -R u.
    \] 
We define $S$ as:
    \begin{align}
        &(S u)_a:= -\tfrac{1}{2} \sum_{i,j \in \N} \left(\phi_j \partial_b \phi_i (A_{ab}(i,j) - A_{ba}(i,j) \right) + \left(  \phi_j \partial_c \phi_i (A_{ac}(i,j) - A_{ca}(i,j))\right), \\
        &A_{\alpha \beta}(i,j) :=\int \fint_{[x_i,x_j]}  D u_\beta(z) \cdot (x_i-x_j) (y-z)_{\alpha} \dH^1(z) \dmu_{i,j}, \label{def:Aij}
    \end{align}
if $i \neq j$ and $A_{\alpha \beta}(i,i) =0$ for all $i \in \N$. 

\noindent The corrector $R$ is defined via
    \begin{align}
        &(R u)_a := -\sum_{i,j,k \in \N} \phi_k \partial_b \phi_j \partial_c \phi_i B(i,j,k), \\
        &B(i,j,k) := \int \fint_{\Sim(i,j,k)} \tfrac{1}{2} (x_i-x_j) \times (x_j-x_k) \cdot \Bigl(\partial_1 u (y-z)_1 + \partial_2 u (y-z)_2  \\&\hspace{4cm}  +  \partial_3 u (y-z)_3 \Bigr) \dH^{2}(z) \dmu_{i,j,k}, \label{def:Bijk}
     \end{align}
if $i$, $j$ and $k$ are pairwise disjoint and    $B(i,i,k)=B(i,j,i)=B(i,j,j)=0$. 
\end{defi}

In this section we prove the following statement: 

\begin{thm} \label{sec3:main}
    If $u \in W^{1,p}(\R^3,\R^3)$ and $\lambda>0$, set 
        \[
        X = X_{\lambda} = \{ \M u \leq \lambda \} \cap \{ \M (D u) \leq \lambda \}.
        \]
    Then the truncated $T u$ defined in Definition \ref{defi:trunc} satisfies all the assertions of Theorem \ref{thmintro:A} with $L=\tfrac{\lambda}{2}$. 
\end{thm}
    
To this end, in Section \ref{sec:proofs}, we proof the following auxiliary results:

\begin{lem} \label{lem:divfree}
On the `bad set' $X^C$, we have $T u \in C^{\infty}(X^C,\R^3)$. Moreover, the strong derivative satisfies for all $y \in X^C$
    \[
     \divergence T u(y) = \divergence T_0 u(y) + \divergence S u(y)+ \divergence R u(y)=0
     \]
whenever $\divergence u =0$. 
\end{lem}
This lemma is a pure calculation, but depends on certain cancellations and use of Stokes' theorem. For the readers convenience, we write down the divergence of the different summands of $T$. The detailed computation can be found in Section \ref{sec:proofs}. For the divergences of the single terms we have 
    \begin{align*}
        \divergence T_0 u(y) &= \sum_{a=1}^3 \sum_{i,j \in \N} \phi_j \partial_a \phi_i \int \fint_{[x_i,x_j]} Du_a(z) \cdot (x_i-x_j) \dH^{1}(z) \dmu_{i,j},\\
        \divergence S u(y) &= - \divergence T_0(y) - \sum_{i,j \in \N} \Bigl[ \bigl(\partial_1 \phi_j \partial_2 \phi_i (A_{12}(i,j) - A_{21}(i,j))\bigr)\\& + \bigl(\partial_2 \phi_j \partial_3 \phi_i (A_{23}(i,j) - A_{32}(i,j))\bigr) +\bigl(\partial_3 \phi_j \partial_1 \phi_i (A_{31}(i,j) - A_{13}(i,j))\bigr)\Bigr], \\
        \divergence Ru (y) &= -(\divergence T_0(y) + \divergence Su(y)).
    \end{align*}

The next lemma shows that, indeed $T$ is a well-defined truncation operator and maps into Lipschitz functions.

\begin{lem} \label{lem:Lipschitz}
Let $u \in W^{1,p}(\R^3,\R^3)$ and $X=X_{\lambda}$. Then we have 
    \begin{enumerate}[label=(\alph*)]
        \item \label{lem:Lipschitz:item1} $T_0 u \in W^{1,\infty}(\R^3,\R^3)$ and $\Vert T_0 u \Vert_{W^{1,\infty}} \leq C \lambda$;
        \item \label{lem:Lipschitz:item2}$Su,~Ru \in W^{1,1}_0(X^C,\R^3)$, in particular the sums in the definition of $R$ and $S$ converge absolutely in $W^{1,1}$;
        \item \label{lem:Lipschitz:item3}$Su,~Ru \in W^{1,\infty}(\R^3,\R^3)$ and $\Vert S u \Vert_{W^{1,\infty}} + \Vert R u \Vert_{W^{1,\infty}} \leq C \lambda$.
    \end{enumerate}
\end{lem}

The last ingredient towards showing Theorem \ref{sec3:main} is to bound the measure of the `bad set' $X_{\lambda}^C$, cf. \cite[Lemma 3.1]{Zhang}.

\begin{lem} \label{lem:badset}
Let $u \in W^{1,p}(\R^3,\R^3)$, $1 \leq p < \infty$. Then we may estimate
    \begin{equation} \label{est:badset}
        \LL^3(X_{\lambda}^C) \leq C \lambda^{-p} \int_{\{ \vert u \vert \geq \lambda/2\} \cup \{\vert D u \vert \geq \lambda/2\}} \vert u \vert^p + \vert D u \vert^p \dx.
    \end{equation}
\end{lem}
We summarise Lemmas \ref{lem:divfree}, \ref{lem:Lipschitz} and \ref{lem:badset} to obtain the proof of Theorem \ref{sec3:main}:

\begin{proof}[Proof of Theorem \ref{sec3:main}]
Lemma \ref{lem:divfree} and Lemma \ref{lem:Lipschitz} show that $Tu \in W^{1,\infty}(\R^3,\R^3)$ and that $\divergence Tu =0$. Furthermore, it gives the $W^{1,\infty}$-bound: 
    \[
    \Vert T u \Vert_{W^{1,\infty}} \leq C \lambda \leq CL.
    \] 
By definition, $u$ and $Tu$ coincide on the set $X_{\lambda}$ and, therefore,
    \[
    \LL^3( \{ u \neq T u \}) \leq \LL^3 ( X_{\lambda}^C) \leq C L^{-p} \LL^3(X_{\lambda}^C) \leq C \lambda^{-p} \int_{\{ \vert u \vert \geq L\} \cup \{\vert D u \vert \geq L\}} \vert u \vert^p + \vert D u \vert^p \dx.
    \]
This establishes property \ref{A:4} in Theorem \ref{thmintro:A}. It remains to show \ref{A:3}. This directly follows from \ref{A:4}, as 
    \begin{align*}
        \Vert u - T u \Vert_{W^{1,p}}^p &\leq \int_{ \{u \neq Tu \}} \vert u - T u \vert^p + \vert D u - D(Tu) \vert^p \dx \\
        &\leq C \int_{X_{\lambda}^C}\left( \vert u \vert +\vert Du \vert^p\right) + \left( \vert Tu \vert^p + \vert D(Tu) \vert^p \right) \dx \\
        &\leq C L^{-p} \int_{\{ \vert u \vert \geq L\} \cup \{\vert D u \vert \geq L\}} \vert u \vert^p + \vert D u \vert^p \dx + \Vert Tu \Vert_{W^{1,\infty}} \LL^3(X_{\lambda}^C) \\
        &\leq C \int_{\{ \vert u \vert \geq L\} \cup \{\vert D u \vert \geq L\}} \vert u \vert^p + \vert D u \vert^p \dx.
    \end{align*}
\end{proof}

\subsection{Intermezzo: Stokes' theorem} \label{sec:intermezzo}
In this section, we recall some versions of Stokes' theorem, which appear in the calculations verifying Lemma \ref{lem:divfree}. Stokes' theorem (which, in general is formulated for differential forms) appears in three instances: as the fundamental theorem, rewriting $u(x_i)-u(x_j)$; as the classical $3D$-Stokes, replacing a boundary integral of a triangle; and in the form of Gau{\ss}-Green.

Usually, one applies Stokes' theorem in situations, where $M$ is an $r$-dimensional manifold, $r=1,2,3$ with sufficiently smooth boundary. For sufficiently smooth functions this assumption can be replaced by Lipschitz boundary, e.g. the boundary of a simplex. 

For vectors $v,w \in \R^3$, recall that the cross-product is defined via 
    \[
        v \times w = \left( \begin{array}{c} v_2 w_3 -w_3 v_2 \\ v_3 w_1 - v_1 w_3 \\ v_1 w_2 - v_2 w_1 \end{array} \right),
    \]
and that for $f \in C^1(\R^3,\R^3)$ we define the rotation $\curl f \in C(\R^3,\R^3)$ as 
    \[
        \curl f= \nabla \times f = \left( \begin{array}{c} \partial_2 f_3 -\partial_3 f_2 \\ \partial_3 f_1 - \partial_1 f_3 \\ \partial_1 f_2 - \partial_2 f_1 \end{array} \right).
    \]
Denote by $\tau$ the tangential vector to a one-dimensional manifold and by $\nu$ the normal to a two-dimensional manifold (the orientation is given by below lemma). Further denote by $I_2$ the index set $I_2=\{(1,2),(2,3),(3,1)\}$ and the index set $I_3$ by $I_3= \{(1,2,3),(2,3,4),(3,4,1),(4,1,2)\}$.    
\begin{lem}[Stokes' theorem applied on simplices]
    Let $u \in C^1(\R^3)$ and $v,w \in C^1(\R^3,\R^3)$ and let $x_1, x_2, x_3, x_4 \in \R^3$ such that $\Sim(1,2,3,4)$ is non-degenerate. Then we have the following:
       \begin{align}
           u(x_i)-u(x_j) &= \fint_{[x_i,x_j]} Du(z) \cdot (x_i-x_j) \dH^1(z) \label{eq:fund} \tag{FT}\\
           \int_{\partial \Sim(1,2,3)} \tau(z) v(z) \dH^1(z) &= \sum_{(a,b) \in I_2} \fint_{[x_a,x_b]} v(z) \cdot (x_a-x_b) \dH^1(z)\nonumber\\
           &= \int_{\Sim(1,2,3)} \curl v(z) \cdot \nu(z) \dH^2(z) \nonumber \\
           &= \fint_{\Sim(1,2,3} \left(\tfrac{1}{2} (x_1-x_2) \times (x_2-x_3)\right) \cdot \curl v(z) \dH^2(z). \label{eq:Stokes} \tag{St}\\
           \int_{\partial \Sim(1,2,3,4)} \nu(z) \cdot w(z) \dH^2(z) &= \sum_{(a,b,c) \in I_3} \fint_{\Sim(a,b,c)}  \left(\tfrac{1}{2} (x_a-x_b) \times (x_b-x_c)\right) \cdot w(z) \dH^2(z)  \nonumber \\
           &= \int_{\Sim(1,2,3,4)} \divergence w(z) \dH^2(z) \cdot \sgn(i,j,k,l). \label{eq:Gauss} \tag{GG}
        \end{align}
        where $\sgn(i,j,k,l) \in \{\pm 1\}$ is a suitable sign function denoting the orientation of the simplex.
\end{lem}
Stokes' theorem (in the form \eqref{eq:fund}) is applied to $u_a(x_i)-u_a(x_j)$ straightforwardly, possibly by a density argument.
\begin{lem}[Application of Stokes' theorem  to the truncation] \label{Stokesappl}
Let $u \in W^{1,p}(\R^3,\R^3)$ satisfying $\divergence u=0$, $i,j,k,l \in \N$. Then, we have
    \begin{align} 
        &\left(A_{ab}(i,j)-A_{ab}(i,k)-A_{ab}(k,i)\right)-\left(A_{ba}(i,j)-A_{ba}(i,k)-A_{ba}(k,j)\right) \nonumber
        \\&= \int \fint_{\Sim(i,j,k)} \tfrac{1}{2}(x_i-x_j)\times(x_j-x_k) \cdot \partial_c u \dH^2(z) \dmu_{i,j}\label{Stokes:A}
\end{align}
and
    \begin{align} \label{Stokes:B}
        B(i,j,k) -B(i,j,l) - B(i,l,k) - B(l,j,k) =0.
    \end{align}
\end{lem}

\begin{proof}
    By a density argument it suffices to consider $u \in C^2(\R^3,\R^3)$ that is divergence-free, such $u$ can be obtained by convolution with mollifiers.
    
    By \eqref{eq:Stokes}, for \eqref{Stokes:A} and antisymmetry  $A_{ab}(ij)= -A_{ab}(ji)$, it suffices to compute
        \[
            \curl_z (D u_{b} (y-z)_a) \quad \text{and} \quad \curl_z (Du_{a} (y-z)_b).
        \]
    Let us take $a=1$, $b=2$ and $c=3$, the statement for other tuples follows by symmetry considerations. Then:
        \begin{align*}
            \curl_z (Du_2(y-z)_1) &= \bigl(\curl_z (Du_2) \bigr) (y-z)_1 + Du_2 \times D_z((y-z)_1) \\
            &=\left( \begin{array}{c} 0 \\ \partial_3 u_2 \\ - \partial_2 u_2\end{array} \right)
        \end{align*}
    The same computation for $\curl_z (Du_1(y-z)_2$ yields for divergence-free $u$:
         \begin{equation} \label{comp:1}
            \curl_z (Du_2(y-z)_1 - Du1(y-z)_2) = \left( \begin{array}{c} \partial_3 u_1 \\ \partial_3 u_2 \\ - \partial_2 u_2 - \partial_1 u_1 \end{array} \right) = \left( \begin{array}{c} \partial_3 u_1 \\ \partial_3 u_2 \\ \partial_3 u_3 \end{array} \right) = \partial_3 u.
        \end{equation}
    Hence, a direct application of \eqref{eq:Stokes} yields \eqref{Stokes:A}.
    
    For \eqref{Stokes:B}, by \eqref{eq:Gauss}, it is enough to compute the divergence of the integrand featured in the definition of $B(i,j,k)$, \eqref{def:Bijk}, i.e.
        \[
            \divergence_z( \sum_{a=1}^3 \partial_a u (y-z)_a)
        \]
    This is a straightforward computation. Note that 
        \begin{align*}
            \partial_1 ( \sum_{a=1}^3 \partial_a u_1 (y-z)_a) = \sum_{a=1}^3 \partial_1 \partial_a u_1 (y-z)_a - \partial_1 u_1.
        \end{align*}
    Therefore,
        \begin{align} \label{comp:2}
            \divergence_z( \sum_{a=1}^3 \partial_a u (y-z)_a) = \sum_{a=1}^3 \bigl[ \partial_a (\divergence u) (y-z)_a \bigr]- \divergence u =0
        \end{align}
    due to $\divergence u=0$. Consequently, the applciation of \eqref{eq:Gauss} directly gives \eqref{Stokes:B}.
\end{proof}

\subsection{Detailed proofs of Lemmas \ref{lem:divfree}-\ref{lem:badset}} \label{sec:proofs}

In this section, we present the proofs of the auxiliary lemmas that are used in the proof of Theorem \ref{sec3:main} (and therefore in the proof of Theorem \ref{thmintro:A}).

We start with showing that $T u$ is divergence free.

\begin{proof}[Proof of Lemma \ref{lem:divfree}]
Note that each summand appearing in the definitions $T$, $S$ and $R$ is in $C^{\infty}$ and is supported on $Q_i$ (and $Q_i \cap Q_j$, $Q_i \cap Q_j \cap Q_k$ for $S u$ and $R u$, respectively). Due to the properties of the cover, therefore only a finite number of terms are nonzero in a neighbourhood of $y \in X^C$. Therefore, $Tu \in C^{\infty}(X_{\lambda}^C)$ and we may compute the pointwise derivative in the classical sense.

\smallskip

\noindent \textbf{Computation of $\divergence T_0 u$:} 
    For simplicity, we compute $\partial_a (T_0 u)_a$ and then give the result for $\divergence T_0 u$:
        \begin{align*}
            \partial_a (T_0 u)_a &= \partial_a \sum_{i \in \N}  \phi_i(y) \int u_a(x_i) \dmu_i(x_i)\\
            &= \sum_{i \in \N} \partial_a \phi_i(y) \int u_a(x_i) \dmu_i(x_i) \\
            &= \sum_{i,j \in \N} \phi_j(y) \partial_a \phi_i(y) \left(\int u_a(x_i) \dmu_i(x_i) - \int u_a(x_j) \dmu_i(x_j) \right)  \\
            &= \sum_{i,j \in \N} \phi_j(y) \partial_a \phi_i(y) \int u_a(x_i) - u_a(x_j) \dmu_{i,j} \\
            & = \sum_{i,j \in \N} \phi_j(y) \partial_a \phi_i(y) \fint_{[x_i,x_j]} D u_a(z) (x_i-x_j) \dmu_{i,j}.
        \end{align*}
In particular we used that $\phi_i$ and $\phi_j$ are partitions of unity, $\mu_i$ and $\mu_j$ are probability measures and the fundamental theorem of calculus. Now, adding up all these terms for $a=1,2,3$ directly gives the divergence of $T_0 u$.

\smallskip

\noindent \textbf{Computation of $\divergence S u$:} 

We compute $\partial_a (Su)_a$ once. We then add up the terms $\partial_1 (Su)_1$, $\partial_2 (Su)_2$ and $\partial_3 (Su)_3$ and use some cancellations.

Note that
    \begin{align*}
         \partial_a A_{ab}(i,j) &= \int \fint_{[x_i,x_j]} Du_b(z) \cdot (x_i-x_j) \dH^1(z) \dmu_{i,j} \\
         \partial_a A_{ba}(i,j) &=0,
    \end{align*}
and, likewise for $A_{ac}$ and $A_{ca}$. Therefore,
    \begin{align*}
        \partial_a (Su)_a = -\frac{1}{2} & \Bigl[ \sum_{i,j \in \N} (\phi_j \partial_a \partial_b \phi_i +\partial_a \phi_j \partial_b \phi_i (A_{ab}(i,j) - A_{ba}(i,j)) \\
        & +\sum_{i,j \in \N} \phi_j \partial_a \partial_c \phi_i + \partial_a \partial_c \phi_i (A_{ac}(i,j) - A_{ca}(i,j)) \\
        &+\sum_{i,j \in \N} \phi_j \partial_b \phi_i  \partial_a A_{ab}(i,j) +   \phi_j \partial_c \phi_i  \partial_a A_{ac}(i,j)\Bigr].
    \end{align*}
Summing $\partial_1 (Su)_1 + \partial_2 (Su)_2 + \partial_3 (Su)_3$, note that terms featuring two derivatives of $\phi_i$, e.g. $\partial_1 \partial_2 \phi_i$, cancel out due to an antisymmetry argument. On the other hand, note that the terms $A_{ab}(i,j)-A_{ba}(i,j)$ are antisymmetric, both when exchanging $a$ and $b$ and when exchanging $i$ and $j$. Therefore, we conclude: 
    \begin{align*}
        \divergence Su &= \partial_1 (Su)_1 + \partial_2 (Su)_2 + \partial_3 (Su)_3 \\
        &= \Bigl[- \sum_{i,j \in \N} \partial_1 \phi_j \partial_2 \phi_i (A_{12}(i,j) - A_{21}(i,j)) - \sum_{i,j \in \N} \partial_2 \phi_j \partial_3 \phi_i (A_{23}(i,j) - A_{32}(i,j)) \\
        & - \sum_{i,j \in \N} \partial_1 \phi_j \partial_2 \phi_i (A_{12}(i,j) - A_{21}(i,j))\Bigr] \\
        &+\Bigl[ -\sum_{i,j \in \N} \phi_j \partial_1 \phi_i \int \fint Du_1(z) \cdot(x_i-x_j) \dH^1(z) \dmu_{i,j} \\
        & -\sum_{i,j \in \N} \phi_j \partial_2 \phi_i \int \fint Du_2(z) \cdot(x_i-x_j) \dH^1(z) \dmu_{i,j} \\&-\sum_{i,j \in \N} \phi_j \partial_3 \phi_i \int \fint Du_3(z) \cdot(x_i-x_j) \dH^1(z) \dmu_{i,j} \Bigr]
    \end{align*}
    
The second part of this sum is simply $- \divergence T_0 u$. Moreover, we can rewrite the first summand using a similar trick we used in the calculations for $\divergence T_0 u$:
    \begin{align*}
        &\sum_{i,j \in \N} \partial_1 \phi_j \partial_2 \phi_i A_{12}(i,j) = \sum_{i,j,k \in \N} \phi_k \partial_1 \phi_j \partial_2 \phi_i (A_{12}(i,j)-A_{12}(i,k)-A_{12}(k,j)).
     \end{align*}
    
One may now apply Stokes' theorem (cf. Lemma \ref{Stokesappl}) to obtain
    \begin{align*}
       & \sum_{i,j \in \N} \partial_1 \phi_j \partial_2 \phi_i (A_{12}(i,j) - A_{21}(i,j)) \\&= \sum_{i,j,k \in \N} \phi_k \partial_1 \phi_j \partial_2 \phi_i \int \fint_{\Sim(i,j,k)} \tfrac{1}{2} (x_i-x_j) \times (x_j-x_k) \cdot \partial_3 u \dH^2(z) \dmu_{i,j,k}. 
    \end{align*}
 Therefore, we get as an alternative formula for $\divergence Su$: 
    \begin{align}
        \divergence Su &= \sum_{i,j,k \in \N} \phi_k \partial_1 \phi_j \partial_2 \phi_i C_3(i,j,k) + \phi_k \partial_2 \phi_j \partial_3 \phi_i C_1(i,j,k) + \phi_k \partial_3 \phi_j \partial_i \phi_i C_2(i,j,k) \nonumber \\
        & - \divergence T_0 u, \label{eq:divS}
     \end{align}   
where, for $a \in \{1,2,3\} $ 
     \begin{equation}
         C_a(i,j,k) = \int \fint_{\Sim(i,j,k)} \tfrac{1}{2} (x_i-x_j) \times (x_j-x_k) \partial_a u \dH^2(z) \dmu_{i,j,k}.
     \end{equation}
  
\smallskip

\noindent \textbf{Computation of $\divergence R u$:} 
Again, we first compute $\partial_a (Ru)_a$ and then combine these and use cancellations. Note that 
    \[
        \partial_a B(i,j,k)(y) = \int \fint_{\Sim(i,j,k)}  \tfrac{1}{2} (x_i-x_j) \times (x_j-x_k) \cdot \partial_a u \dH^2(z) \dmu_{i,j,k} = C_a(i,j,k).
    \]
 Therefore,
    \begin{align*}
        \partial_a (R u)_a &=- \sum_{i,j,k \in \N} \Bigl( \partial_a \phi_k \partial_b \phi_j \partial_c \phi_i + \phi_k \partial_a \partial_b \phi_j \partial_c \phi_i + \phi_k \partial_b \phi_j \partial_c \partial_a \phi_i\Bigr) B(i,j,k) \\
        & - \sum_{i,j,k \in \N} \phi_j \partial_b \phi_j \partial_c \phi_i C_a(i,j,k).
    \end{align*}
When we compute the divergence, terms with two derivatives on $\phi_j$ and $\phi_i$ cancel out due to symmetry considerations: $B(i,j,k)$ is anti-symmetric in $i$, $j$ and $k$. In particular, 
    \[
        B(i,j,k) = B(j,k,i) = B(k,i,j) = -B(j,i,k) = - B(i,k,j) = -B(k,j,i).
    \]
This yields: 
    \begin{align*}
         \divergence Ru &= - 3 \sum_{i,j,k \in \N} \partial_1 \phi_k \partial_2 \phi_j \partial_3 \phi_i B(i,j,k) \\ &- \left(\sum_{i,j,k \in \N} 
        \phi_k \partial_1 \phi_j \partial_2 \phi_i C_3(i,j,k) + \phi_k \partial_2 \phi_j \partial_3 \phi_i C_1(i,j,k) + \phi_k \partial_3 \phi_j \partial_i \phi_i C_2(i,j,k) \right) \\
         &=  - 3 \sum_{i,j,k \in \N} \partial_1 \phi_k \partial_2 \phi_j \partial_3 \phi_i B(i,j,k) - (\divergence T_0 u + \divergence Su).
        \end{align*}
For the first summand, again we can use that $\phi_i$ is a partition of unity to introduce another index $l$: 
    \begin{align*}
        &- 3 \sum_{i,j,k \in \N} \partial_1 \phi_k \partial_2 \phi_j \partial_3 \phi_i B(i,j,k) \\
        &= -3 \sum_{i,j,k,l \in \N} \phi_l \partial_1 \phi_k \partial_2 \phi_j \partial_3 \phi_i (B(i,j,k)-B(l,j,k)-B(i,l,k)-B(i,j,l)).
    \end{align*}
 But this equals zero due to Stokes' theorem, cf. Lemma \ref{Stokesappl}. Therefore,
    \[
    \divergence Ru = - (\divergence T_0 u + \divergence Su)
    \]
and Lemma \ref{lem:divfree} is proven.    
\end{proof}
\begin{rem}
It is `hidden' in the computation above, why we restrict to $N=3$ for simplicity. We apply Stokes' theorem for $r$-dimensional simplices successively  for $r=1$ (the fundamental theorem), $r=2$ (`classical' 3D Stokes) and $r=3$ (Gau{\ss}-Green theorem). In higher dimension, one needs to apply Stokes' theorem $N$-times.  A similar calculation for lower regularity in that setting is carried out in \cite{Schiffer}.

\end{rem}
The next step is to show Lemma \ref{lem:Lipschitz}. A key point in Lemma \ref{lem:Lipschitz} is to find a good estimate on $A_{ab}$ and $B$. A key step of this estimate is the following lemma:

\begin{lem} \label{lem:help}
Let $Q_1=[-1/2,1/2]^N$, $Q_2$ be a cube of sidelength $1/4 \leq l(Q_2) \leq 4$ with $Q_2 \subset B_{10}(0)$. Suppose that $v \in L^1(B_{10}(0))$. Then:
\begin{enumerate} [label=(\alph*)]
    \item We can estimate 
        \begin{equation} \label{lem:help:1}
            \int_{Q_1} \int_{Q_2} \int_{[x_1,x_2]} \vert v(z) \vert \dH^1(z) \dx_j \dx_i \leq  C \int_{B_{10}(0)} \vert v(z) \vert \dL^3(z).
        \end{equation}        
    \item If, in addition, $Q_3$ is a cube with sidelength $1/4 \leq l(Q_3) \leq 4$, $Q_3 \subset B_{10}(0)$, then we may estimae
        \begin{equation} \label{lem:help:2}
            \int_{Q_1} \int_{Q_2} \int_{Q_3} \int_{\Sim(x_1,x_2,x_3)} \vert v(z) \vert \dH^1(z) \dx_j \dx_i \leq  C \int_{B_{10}(0)} \vert v(z) \vert \dL^3(z).
        \end{equation}
\end{enumerate}
\end{lem}
The proof of Lemma \ref{lem:help} is an application of Fubini/transformation rules. First, we see how this Lemma now shows the desired Lemma \ref{lem:Lipschitz}.

\begin{proof}[Proof of Lemma \ref{lem:Lipschitz}]
The first claim \ref{lem:Lipschitz:item1} follows from the fact that $T_0 u$ is a Lipschitz truncation, cf. Lemma \ref{lem:Lipschitztrunc} and \cite{Zhang}.

For \ref{lem:Lipschitz:item2} and \ref{lem:Lipschitz:item3}, we estimate the values of single summands. We start with $(Su)$. Recall that 
    \begin{align} \label{est:phi}
        \Vert \phi_i \Vert_{L^{\infty}} \leq 1,\quad \Vert D \phi_i \Vert_{L^{\infty}} \leq C l(Q_i)^{-1}, \quad \Vert D^2 \phi_i \Vert_{L^{\infty}} \leq C l(Q_i)^{-2}.
    \end{align}
Moreover, if $Q_i \cap Q_j \neq \emptyset$, the length $l(Q_i)$ and $l(Q_j)$ are comparable. Whenever $y \in Q_i$, we have $\vert y - z \vert \leq C l(Q_i)$ in below integrand and, therefore
    \begin{align}
        \vert A_{\alpha \beta}(i,j) (y) \vert& \leq  \int \int_{[x_i,x_j]} \vert Du_{\beta}(z)\vert \vert y- z \vert_{\alpha} \dH^1(z) \dmu_{i,j} \nonumber
        \\
        &\leq C l(Q_i) \cdot \frac{1}{l(Q_i)^6} \int_{1/2 Q_i^{\ast}} \int_{1/2 Q_j^{\ast}} \int_{[x_i,x_j]} \vert D u (z) \vert \dH^1(z) \dx_j \dx_i \nonumber
        \\
        &= \frac{C}{l(Q_i)^5} \int_{1/2 Q_i^{\ast}} \int_{1/2 Q_j^{\ast}} \int_{[x_i,x_j]} \vert D u(z) \vert \dH^1(z) \dx_j  \dx_i \label{calc:1}
    \end{align}
 Note that the integral scales with power $3+3+1=7$ in $l(Q_i)$ (as $l(Q_i)$ and the length of $[x_i,x_j]$ are comparable to $l(Q_i)$). Therefore, applying Lemma \ref{lem:help}, we get
    \begin{align*}
         \int_{1/2 Q_i^{\ast}} \int_{1/2 Q_j^{\ast}} \int_{[x_i,x_j]} \vert D u(z) \vert \dH^1(z) \dx_j  \dx_i &\leq C l(Q_i)^7 \fint_{B_{10l(Q_i)}(c_i)} \Vert Du (z) \vert \dL^3(z) \\& \leq Cl(Q_i)^7 \M (Du)(z_i) \leq Cl(Q_i)^7 \lambda.
    \end{align*}
Conbining this with \eqref{calc:1} yields the estimate 
    \[
        \vert A_{\alpha \beta} (i,j) (y) \vert \leq C \lambda l(Q_i)^2.
    \]
The same calculation for the derivative of $A_{\alpha \beta}$ instead of $A_{\alpha \beta}$ itself (Note that there we miss the additional decay through $\vert y- z \vert^{\alpha}$):    
    \begin{equation}
      \vert  \partial_{\gamma} A_{\alpha \beta}(i,j) \vert \leq C \lambda  l(Q_i).
    \end{equation}
Using this, we may estimate 
    \begin{align*}
         \Vert \phi_j \partial_b \phi_i A_{ab}(i,j) \Vert_{L^{\infty}}& \leq C \lambda, \\
         \Vert D( \phi_j \partial_b \phi_i A_{ab}(i,j) \Vert_{L^{\infty}} & \leq \Vert D \phi_j \partial_b  \phi_i A_{ab}(i,j) \Vert_{L^{\infty}} + \Vert \phi_j D \partial_b \phi_i A_{ab}(i,j) \Vert_{L^{\infty}} \\& \quad + \Vert \phi_j \partial_b \phi_i D A_{ab}(i,j) \Vert_{L^{\infty}} \\ &\leq C \lambda
    \end{align*}
and we conclude that single summands in the definition of $Su$ have $W^{1,\infty}$-norm bounded by some uniform constant $C$.

Consequently, for any $i,j$, $\phi_j \partial_b \phi_i A_{ab}(i,j) \in C_c^{\infty}(X^C)$ and we have the estimate
    \[
        \Vert \phi_j \partial_b \phi_i A_{ab}(i,j) \Vert_{W^{1,1}} \leq \LL^N(Q_i \cap Q_j) C \lambda.
    \]
A point $y \in X^C$ is contained in at most $C_3$ cubes for a dimensional constant $C_3$. Therefore,
    \[
        \sum_{i,j \in \N} \Vert \phi_j \partial_b \phi_i A_{ab}(i,j) \Vert_{W^{1,1}} \leq \sum_{i,j \in \N} \LL^N(Q_i\cap Q_j) C \leq C_3^1 \LL^3(X^C) C \lambda.
    \]
As $\LL^3(X^C)$ is bounded (cf. Lemma \ref{lem:badset}), we establish absolute convergence in $W^{1,1}$ and, therefore, $Su \in W^{1,1}_0(X^C)$. The $W^{1,\infty}$-bound of $Su$ then follows by the observation, that for any $y \in X^C$, only $C_3^2$ terms might be non-zero. Therefore,
    \[
        \Vert \sum_{i,j \in \N} \phi_j \partial_b \phi_i A_{ab}(i,j) \Vert_{W^{1,\infty}} \leq C_3^2 \lambda,
    \]
and the same bound follows for $\Vert Su  \Vert_{W^{1,\infty}}$.

\smallskip

Now, we bound the $W^{1,\infty}$-norm of $(Ru)$. To this end, whenever $y \in Q_i \cap Q_j \cap Q_k \neq \emptyset$, similar to  \eqref{calc:1}, we have the estimate 
    \begin{align}
        \vert B(i,j,k) \vert &\leq \int \int_{\Sim(i,j,k)} \vert Du(z) \vert \vert y -z\vert \dH^2(z) \dmu_{i,j,k} \nonumber
        \\
        &\leq \frac{C}{l(Q_i)^8} \int_{Q_i} \int_{Q_j} \int_{Q_k}  \int_{\Sim(i,j,k)} \vert Du(z) \vert \dH^2(z) \dx_k \dx_j \dx_i.  \label{calc:2}
     \end{align}
Note that the integral scales with power $3+3+3+2=11$ in the sidelength of the cube $(Q_i)$; again note that the sidelength of $Q_j$ and $Q_k$ are comparable to $l(Q_i)$. Therefore, applying the second part of Lemma \ref{lem:help}:    
     \begin{align*}
         \int_{Q_i} \int_{Q_j} \int_{Q_k}  \int_{\Sim(i,j,k)} \vert Du(z) \vert \dH^2(z) \dx_k \dx_j \dx_i &\leq C l(Q_i)^{11} \int_{B_{10l(Q_i)}(c_i)} \vert Du(z) \vert \dz \\
         &\leq Cl(Q_i)^{11} \M(Du)(x_i) \leq Cl(Q_i)^{11} \lambda
     \end{align*}
Combining this estimate with \eqref{calc:2} gives 
  \begin{align} \label{calc:22}
        \vert B(i,j,k) \vert \leq Cl(Q_i)^3 \lambda.
    \end{align} 
The same argumentation for the derivative of $B(i,j,k)$ instead also gives \begin{equation} \label{calc:3}
    \vert D_y B(i,j,k) \vert \leq C l(Q_i)^2 \lambda.
\end{equation}   
We conclude by \eqref{calc:22} and \eqref{est:phi} that for the $L^{\infty}$-norm of a single summand we have the estimate \begin{equation}
    \Vert \phi_k \partial_b \phi_j \partial_c \phi_i B(i,j,k) \Vert_{L^{\infty}} \leq C \lambda.
\end{equation}
Moreover, for the derivative of a single summand we get by \eqref{calc:22}, \eqref{est:phi} and \eqref{calc:3}
    \begin{align}
        \Vert D(\phi_k \partial_b \phi_j \partial_c \phi_i B(i,j,k)\Vert_{L^{\infty}}& \leq \Vert D\phi_k \partial_b \phi_j \partial_c \phi_i B(i,j,k) \Vert_{L^{\infty}} +
        \Vert \phi_k \partial_b D\phi_j \partial_c \phi_i B(i,j,k) \Vert_{L^{\infty}} \nonumber \\
        &+ \Vert \phi_k \partial_b \phi_j D\partial_c \phi_i B(i,j,k) \Vert_{L^{\infty}} + \Vert \phi_k \partial_b \phi_j \partial_c \phi_i DB(i,j,k) \Vert_{L^{\infty}} \nonumber \\
        &\leq C \lambda. \label{est:3}
    \end{align}
    
Therefore, we can infer for the $W^{1,1}$ norm of a nonzero summand: 
    \begin{equation*}
        \Vert \phi_k \partial_b \phi_j \partial_c \phi_i B(i,j,k) \Vert_{W^{1,1}} \leq \LL^3(Q_i \cap Q_j \cap Q_k) C \lambda.
    \end{equation*}
As any $y \in Z^C$ is contatined in at most $C_3$ cubes $Q_j$, it is at most contained in $C_3^3$ sets of the form $Q_i \cap Q_j \cap Q_k$, and therefore:
    \begin{equation} \label{R:absconv}
        \sum_{i,j,k \in \N} \Vert \phi_k \partial_b \phi_j \partial_c \phi_i B(i,j,k) \Vert_{W^{1,1}} \leq C_3^3 \LL^3(X^C) C \lambda.
    \end{equation}
This establishes absolute $W^{1,1}$-convergence, and thus $(Ru) \in W^{1,1}_0(X^C)$. As the sum is finite locally with at most $C_3^3$ nonzero summands, we then also conclude
    \begin{equation} \label{R_bound}
        \Vert R u \Vert_{W^{1,\infty}} \leq C_3^2 C \lambda.
    \end{equation}
   
\end{proof}
We now prove Lemma \ref{lem:help}.

\begin{proof}
We first prove \eqref{lem:help:1}. We simply set $v \equiv 0$ outside of $B_ {10}(0)$. Using that $Q_2 \subset B_{10}(0)$ we may estimate
\[
 \int_{Q_1} \int_{Q_2} \int_{[x_1,x_2]} \vert v(z) \vert \dH^1(z) \dx_2 \dx_1  \leq \int_{Q_1} \int_{B_{11}(x_1)} \int_0^1 \vert v((1-t)x_1+tx_2) \vert \cdot \vert x_1-x_2 \vert \dt \dx_2 \dx_1.
\]
Fix $x_1 \in Q_1$ and only consider the two inner integrals. We apply polar coordinates and then the transformation $tr =s$:

    \begin{align*}
        &\int_{B_{11}(x_1)}  \int_0^1 \vert v((1-t)x_1+tx_2) \vert \cdot \vert x_1-x_2 \vert \dt \dx_2= \int_0^{11} r^2 \int_{\mathbb{S}^2} \int_0^1 \vert v(x_1+tr\Theta) \vert \dt \dH^2(\Theta) \dr \\
        &\quad = \int_{\mathbb{S}^2} \int_0^{11} \int_0^r \vert v(x_1+s \Theta) s r \ds \dr \dH^2(\Theta) = \int_{\mathbb{S}^2} \int_0^{11} \int_{s}^11 \vert v(x_1+s\Theta) \vert sr \dr \ds \dH^2(\Theta) \\
        & \quad \leq C \int_{\mathbb{S}^2} \int_0^{11} \vert v(x_1+s \Theta) \vert s \ds \dH^2(\Theta) \leq C \int_{B_{11}(x_1)} \frac{\vert v(z) \vert}{\vert z -x_1\vert} \dz.
    \end{align*}
Now computing the integral over $x_1$, we achieve:
    \begin{align*}
         \int_{Q_1}& \int_{Q_2} \int_{[x_1,x_2]} \vert v(z) \vert \dH^1(z) \dx_2 \dx_1\leq  C \int_{Q_1} \int_{B_{11}(x_1)} \frac{\vert v(z) \vert}{\vert z -x_1\vert} \dz \\
         & \leq C \int_{B_{10}(0)} \vert v(z) \vert \int_{Q_1} \vert z-x_1 \vert^{-1} \dx_1 \dz \leq  C \int_{B_{10}(0)} \vert v(z) \vert \int_{B_{10}(0)} \vert y \vert^{-1} \dy \dz \\
         &\leq   C \int_{B_{10}(0)} \vert v(z) \vert \dz.
    \end{align*}
The computation of the right hand side of \eqref{lem:help:2} 
\[
\int_{Q_1} \int_{Q_2} \int_{Q_3} \int_{\Sim(1,2,3)} \vert v(z) \vert \dH^2(z) \dx_3 \dx_2 \dx_1 =(\ast)
\]
is very similar. Fixing $x_1$ and $x_2$, denote by $[x_1,x_2]$ the straight line through $x_1$ and $x_2$. Applying cylindrical coordinates around the axis $[ x_1,x_2 ]$ we find
\[
(\ast) \leq C\int_{Q_1} \int_{Q_2} \int_{B_{10}(0)} \frac{\vert v(z) \vert}{\dist(z,[ x_1,x_2])} \dz \dx_1 \dx_2
\]
Application of Fubini and the estimate 
    \[
     \int_{Q_1} \int_{Q_2} \dist^{-1}(z,[x_1,x_2 ] ) \dx_2 \dx_1 \leq \int_{B_{10}(0)} \int_{B_{10}(0)} \dist^{-1}(0, [ x_1,x_2 ]) \dx_2 \dx_1  \leq C 
    \]
yields
    \[
    (\ast) \leq C \int_{B_{10}(0)} \vert v(z) \vert \dz.
    \]
\end{proof}

It remains to prove Lemma \ref{lem:badset}. The proof is already contained in \textsc{Zhang}'s paper \cite{Zhang}. For the sake of completeness we outline the short argument.
\begin{proof}[Proof]
    It suffices to show that if $v \in L^p(\R^N,\R^d)$, that then 
        \begin{equation} \label{badset:claim}
            \LL^N(\{ \M v > \lambda\}) \leq C \lambda^{-p} \int_{\{\vert v \vert \geq \lambda/2\}} \vert v \vert^p \dx.
        \end{equation}
    To achieve \eqref{badset:claim}, define $\tilde{v}$ as follows: 
        \[
            \tilde{v}= \left\{ \begin{array}{ll} \vert v \vert -\lambda/2 & \vert v \vert \geq \lambda/2 \\ 0 & \vert v \vert \leq \lambda/2
            \end{array} \right.
        \]
    Then $\{\M v >\lambda \} \subset \{ \M \tilde{v} >\lambda/2\}$ by checking the definition of the maximal function. We conclude by the weak-type estimate for the maximal function, Lemma \ref{lem:maximalf},
        \begin{align*}
            \LL^N(\{ \M v > \lambda \}) \leq \LL^N(\M \tilde{v} > \lambda/2) \leq C\lambda^{-p} \Vert \tilde{v}\Vert_{L^p}^p
                                        \leq C\int_{\V v \vert > \lambda/2} \vert v \vert^p \dx.
        \end{align*}    
\end{proof}

\bibliography{biblio_sol.bib}
\bibliographystyle{abbrv}

\end{document}